\documentclass[12pt]{amsart}
\usepackage{amscd, amssymb, amsmath}
\input{xy}
\xyoption{all}
\usepackage{hyperref}

\usepackage{geometry}

\theoremstyle{definition}
\newtheorem{defi}[equation]{Definition}
\newtheorem{thm}[equation]{Theorem}

\newtheorem*{thmrp}{\bf Existence of rational points over finite fields}
 
\newtheorem{lem}[equation]{Lemma}

\newtheorem{cor}[equation]{Corollary}
\newtheorem{rmk}[equation]{Remark}
\newtheorem{change}[equation]{Change-of-variable formula}

\numberwithin{equation}{section}

\newcommand{\la}{\langle}
\newcommand{\ra}{\rangle}

\newcommand{\ul}{\underline}

\title{Pseudonorms and $p$-adic birational Torelli theorem}

\author{Chen-Yu Chi}

\address{Chen-Yu Chi: Department of Mathematics,
National Taiwan University, Taipei 10617, Taiwan}
\email{geometrychi@yahoo.com.tw, chi@math.ntu.edu.tw}

\begin{document}
\fontsize{12pt}{20pt}\selectfont
\maketitle
\setlength{\baselineskip}{18pt}

\begin{abstract} A $p$-adic analogue of the pseudonorm version of the birational Torelli type theorem is obtained via a comparison theorem of image closures. Among other results obtained, we have a criterion for existence of rational points of canonically polarized surfaces over finite fields. 
\end{abstract}

\maketitle

\section{\bf Introduction}

In \cite{chi-yau} and \cite{chi}, Yau and the author have initiated a study of birational equivalence via the pseudonorm functions on pluricanonical spaces of complex projective manifolds and obtained theorems of Torelli type for birational equivalence. In the current paper, we prove a $p$-adic version of the aforementioned result. In order to state the main result we first set up some notation and terminology. Let $K$ be a completion of an algebraic number field at a prime divisor, $R$ its ring of integers, and $\mathbf F_q$ its residue field. Below we call such $K$ a $p$-adic field if $q$ is a power of some prime number $p$. $|\cdot|_K$ will be a fixed non-archimedean absolute value on $K$ compatible with the valuation on $R$. 

We consider an integral $R$-scheme $\xymatrix{X\ar[r]^-\pi&\mathrm{Spec}\,R}$ which is proper and smooth over $\mathrm{Spec}\, R$. Note that $X(K)$, the set of all $K$-points of $X$, is a $K$-analytic manifold\footnote{See \cite{igusa} 2.4 for the definitions and basic properties. $K$-analytic manifolds are assumed to be 2nd countable.}.   We let 
$$
\Gamma_{\mathrm{an}}\big(X(K),K_{X(K)}^{\otimes m}\big)=\left\{\text{all }K\text{-analytic section of }K_{X(K)}^{\otimes m}\text{ on }X(K)\right\},
$$
where $K_{X(K)}=\wedge^{\dim\, X}T^*X(K)$ is the canonical bundle of $X(K)$, and let 
$$
\Gamma(X,\omega_{X/R}^{\otimes m})=H^0(X,\omega_{X/R}^{\otimes m}),
$$
where $\omega_{X/R}=\omega_X\otimes \pi^*\omega_{\mathrm{Spec}\, R}$. We have a natural homomorphism
$$
\xymatrix{
\Gamma(X,\omega_{X/R}^{\otimes m})\ar[r]&\Gamma_{\mathrm{an}}\big(X(K),K_{X(K)}^{\otimes m}\big)
}
$$
by viewing scheme-theoretic sections on $X$ as $K$-analytic sections on $X(K)$. Since all $K$-analytic sections are Borel measurable sections, we have by Definition \ref{pseudonorm} the following maps
$$
\xymatrix{
\Gamma(X,\omega_{X/R}^{\otimes m})\ar[r]&\Gamma_{\mathrm{an}}\big(X(K),K_{X(K)}^{\otimes m}\big)\ar[r]^-{\la\!\la\cdot\ra\!\ra_m}&[0,\infty).
}
$$
Finally, if $V$ is a nonzero $R$-submodule of $\Gamma(X,\omega_{X/R}^{\otimes m})$, we have the restricted $m$-th pluricanonical $R$-rational map\footnote{For $S$-schemes $X$ and $Y$, a rational map $\xymatrix{X\ar[r]^-{f}&Y}$ is said to be $S$-rational if there exists a dense open set $U$ of $X$ which intersects every fibre of $X$ over $S$ such that $\xymatrix{U\ar[r]^{f|_U}&Y}$ is an $S$-morphism; $f$ is $S$-birational if $U$ can be further chosen so that $f(U)$ is open in $Y$ and $f$ maps $U$ onto $f(U)$ isomorphically.}
$$
\xymatrix@C40pt{X\ar@{-->}[r]^-{\varphi_{|V|}}&\mathbf PV^*:=\mathrm{Proj}\, \bigoplus_{l=0}^\infty\mathrm{Sym}_R^l V.}
$$

Now consider two integral schemes $X$ and $Y$ which are proper and smooth over $\mathrm{Spec}\, R$. The following is the first main result of the current paper.

\begin{thm}\label{image} For an $m\in\mathbf N$ and for an isomorphism 
$$ 
\xymatrix{V_Y\ar[r]^-{T} & V_X}
$$
between $R$-submodules $V_X\subseteq\Gamma(X,\omega_{X/R}^{\otimes m})$ and $V_Y\subseteq \Gamma(Y,\omega_{Y/R}^{\otimes m})$ which preserves the pseudonorm $\la\!\la\cdot\ra\!\ra_m$, if $X(K)\neq\emptyset\neq Y(K)$, the induced isomorphism 
$$
\xymatrix{\mathbf P V_X^*\ar[r]^-{\mathbf PT^*}& \mathbf P V_Y^*}
$$
of $R$-schemes identifies the image closures of the $R$-rational maps $\varphi_{|V_X|}$ and $\varphi_{|V_Y|}$.
\end{thm}

We will see later (Proposition \ref{bir inv}) that for every $m\in\mathbf N$ the pseudonormed space $\big(\Gamma(X,\omega_{X/R}^{\otimes m}), \la\!\la\cdot\ra\!\ra_m\big)$ is an $R$-birational invariant of $X$ in the sense that associated to every $R$-birational map $\xymatrix{X\ar@{-->}[r]^-f&Y}$ is an $R$-linear isometry
$$
\xymatrix{\big(\Gamma(Y,\omega_{Y/R}^{\otimes m}), \la\!\la\cdot\ra\!\ra_m\big)\ar[r]^-{f^*}&\big(\Gamma(X,\omega_{X/R}^{\otimes m}), \la\!\la\cdot\ra\!\ra_m\big).}
$$  
Note that if the base change $X_K$ is of general type, then there exists some $m_X\in\mathbf N$ such that $\varphi_{X,m}:=\varphi_{\left|\Gamma(X,\omega_{X/R}^{\otimes m})\right|}$ is an $R$-birational map from $X$ onto the image closure of $\varphi_{X,m}$ for every $m\geqslant m_X$.\footnote{Here one may extend a birational map over $K$ to an $R$-birational map by a classical result of Matsusaka and Mumford (\cite{mm} Theorem 1). See also \cite{kovacs} Theorem 5.13 for a rephrasing close to the current situation.} Consequently, Theorem \ref{image} has the following corollary.
\begin{cor}[{{\bf $p$-adic birational Torelli theorem}}]\label{bir tor} Suppose that $X_K$ and $Y_K$ are of general type over $K$. If $X(K)\neq\emptyset\neq Y(K)$, if for some $m\in\mathbf N$ there is an $R$-linear isometry
$$
\xymatrix{\big(\Gamma(Y,\omega_{Y/R}^{\otimes m}), \la\!\la\cdot\ra\!\ra_m\big)\ar[r]^-{T}&\big(\Gamma(X,\omega_{X/R}^{\otimes m}), \la\!\la\cdot\ra\!\ra_m\big)},
$$  
and if both $\varphi_{X,m}$ and $\varphi_{Y,m}$ maps birationally onto their image closures, then there exists an $R$-birational map $\xymatrix{X\ar[r]^-f&Y}$ and some $c\in R^\times$ such that $T=cf^*$.  
\end{cor}

Our next main result is related to the condition $X(K)\neq\emptyset$. It is well known that, if $X(\mathbf F_q)\neq\emptyset$, then $X(K)=X(R)\neq\emptyset$ by Hensel's lemma and the valuative criterion. For a fixed $X$ and a fixed prime number $p$, clearly $X(\mathbf F_q)\neq\emptyset$ if $q=p^r$ is sufficiently large. It is then natural to ask how large $q$ would be sufficient to imply $X(\mathbf F_q)\neq\emptyset$. Let $n=\dim\, X_{\mathbf F_q}$. When $X_{\mathbf F_q}$ is smooth, in discussions one usually assumes that $X_{\mathbf F_q}$ is geometrically connected, for otherwise $X(\mathbf F_q)=\emptyset$. According to the solution to Weil's conjecture \cite{deligne}, an estimate of $\# X(\mathbf F_q)$ can be obtained in terms of the Betti numbers. To obtain explicit results following this line, one has to have information of the Betti numbers, and this is mostly achievable for completely intersections. For example, if $X_{\mathbf F_q}$ is a complete intersection in $\mathbf P^{n+r}_{\mathbf F_q}$ of multi-degree $(d_1,\dots,d_r)$, Deligne \cite{deligne} has obtained the estimate
$$
\left|\#  X(\mathbf F_q)-q^n-q^{n-1}-\cdots-1\right|\leqslant b_n'(n+r,d_1,\dots,d_r)q^{\frac{n}{2}}
$$
where $b_n'(n+r,d_1,\dots,d_r)q^{\frac{n}{2}}$ is the $n$-th primitive Betti number of any nonsingular complete intersection of $\mathbf P^{n+r}$ of dimension $n$ and multi-degree $(d_1,\dots,d_r)$, which can be written down explicitly in terms of $n$, $r$, and $(d_1,\dots,d_r)$.  Among many explicit results along this line, it is known for ecample that (cf. \cite{cmp} Theorem 1.2) 
$$
X(\mathbf F_q)\neq\emptyset\, \text{ if }\, 
q> 2\Big(2+\sum\nolimits_{j=1}^r(d_j - 1) \Big)^2 (d_1\cdots d_r)^2.
$$
A natural question is then to know what one can say about situations which involve possibly non-complete-intersections. Besides, we look for results on the nontriviality of $X(\mathbf F_q)$ involving intersection numbers of divisors, instead of involving Betti numbers. We have the following result on existence of rational points over $\mathbf F_q$ involving the intersection numbers of the canonical divisor and a very ample divisor.
\begin{thm} \label{nontrivial} Let $W$ be a complete smooth geometrically connected scheme over $\mathbf F_q$ of dimension $n$. Suppose that $H$ is a very ample divisor. Then 
$$
W(\mathbf F_q)\neq\emptyset\ \text{ if }\ q>\max\left\{H^{\bullet n}(H^{\bullet n}-1)^n,\, 
\big(K_W\bullet H^{\bullet (n-1)} + (n-1) H^{\bullet n}+2\big)^2
\right\},
$$
where $\bullet$ denotes the numerical intersection product of divisors.
\end{thm}
One might hope to deduce a result of similar form by applying the Lang--Weil bound \cite{lw}, which considers varieties with an embeding into an ambient projective space. However, the Lang--Weil bound involves the dimension of the ambient projective space. Theorem \ref{nontrivial} does not require knowing such information. Now we specialize to the canonically polarized situation, which provides the most basic objects for the birational Torelli-type theorem. By the fundamental work of \cite{ekedahl} we obtain the following explicit effective result.
\begin{cor}[{{\bf Existence of rational points on canonically polarized surfaces}}]\label{pol sur} For any complete smooth geometrically connected surface $W$ over $\mathbf F_q$ with $K_W$ ample, 
$$W(\mathbf F_q)\neq \emptyset
\ \text{ if }\  
q> \max\left\{25K_W^{\bullet 2}(25K_W^{\bullet 2}-1)^2, \big(30 K_W^{\bullet 2}+2\big)^2\right\}.
$$
\end{cor}
This can be derived by simply taking $H=K_W^{\otimes 5}$ in Theorem\ref{nontrivial}, since $K_W^{\otimes 5}$ is very ample according to \cite{ekedahl}. It seems that so far there is no established generalization of such effective very ampleness result when the dimension is greater than two. In dimension two. On the other hand, one may also consider the more general situation that $W$ be a smooth surface of general type. Explicit effective result of similar kind may be obtained in this more general situation, but it requires more elaborative treatments of singularities of canonical models, which will appear in \cite{chi2}. 

As for the organization of the paper, in Sec.\,2 we briefly review the notion of $p$-adic integration and introduce the pseudonorms. In Sec.\,3 we prove a density theorem (Theorem \ref{dens}) for rational points in the set of geometric points of integral schemes over a complete normed field. Although for later use we only need the density theorem for $p$-adic fields, which can be more directly obtained by using $p$-adic measure theory, we think it is worth proving a more general version the proof of which makes no use of integration. In Sec.\,4 we show that the pseudonormed pluricanonical spaces for a family of $R$-birational invariants. In Sec.\,5 we prove a $p$-adic parallel (Theorem \ref{equim}) to an equimeasurability theorem of Rudin. Rudin's theorem is for the field $\mathbf C$, and the $\mathbf C$-version of the birational Torelli type theorem can be viewed as a special case of Rudin's result, as pointed out to me by S.\ Antonakoudis. See the survey article \cite{yau}. In his proof Rudin quotes Wiener's invariant subspace theorem. We preferred not proceeding detailed examinations of whether Wiener's result has a suitable analogue for $K$; instead we adopt a setting of a flavour of the Hilbert Nullstellensatz and makes more elementary use of the Fourier transformation. A key step in both Rudin's and our proofs is to construct a non-identically $0$ function of particular type. The analysis in the case of $\mathbf C$ and our case of $p$-adic are quite different, the latter being simpler. Finally, in Sec.\,6 we derived the first main theorem from the $p$-adic equimeasurability theorem. In contrast with the case of $\mathbf C$, we have $K\neq\overline{K}$, and further care of density needs to be taken. Finally, in Sec.\,7 we prove Theorem \ref{nontrivial}. The method is to reduce the proof to the case of curves and then apply  the Hasse--Weil bound. In the case over finite fields, this requires finding possibly non-generic but smooth hyperplane sections, and fortunately we found \cite{ballico}.

Finally I like to make some notes about the current paper and work of Shuang-Yen Lee. The first version of the current paper only consisted of Theorem \ref{image} and the $p$-adic birational Torelli theorem, obtained in October, 2020; it did not treat the question about the nontriviality of $X(K)$. Later on April 30, 2021, I happened to see \cite{lee}, which is a draft of an undergraduate thesis by Lee supervised by Chin-Lung Wang. In \cite{lee} Lee obtained the $1$-dimensional case of the $p$-adic birational Torelli theorem, essentially following the arguments of Royden \cite{royden}, which is independent from the approach we have adopted here, the equimeasurability framework of Rudin. The approach of Royden seems hard to be carried out in higher dimensions to yield the most general birational Torelli theorem, at least for now. On May 8, 2021 I informed Chin-Lung Wang and other collegue that earlier I have obtained Theorem \ref{image} and the $p$-adic birational Torelli theorem, with a copy of my first draft. In July 2021, I reported on Theorem \ref{image} and Corollary \ref{bir tor}, as well as the $p$-adic analogue of Rudin's equimeasurablity theorem (Theorem \ref{equim}), in the HAYAMA Symposium on Complex Analysis in Several Variables XXII. Very recently I am aware of a preprint by Lee \cite{lee2}, which is a revised extension of \cite{lee}. Theorem 0.1 of \cite{lee2} reads very similar to our Theorem \ref{image}, except that we have $R$ as the base ring instead of $K$. Actually, assuming Theorem \ref{equim} and replacing all appearances of $R$ by $K$ in our proof of Theorem \ref{image} (Sec.\!\ \ref{proof of image}) one obtains a proof of Theorem 0.1 of \cite{lee2}. Besides, \cite{lee2} contains essentially a proof of my Theorem {\ref{equim}. Both proofs of the $p$-adic equimeasurability result models on Rudin's setting. Rudin's proof of equimeasurability theorem over $\mathbf C$ uses properties of Fourier transform and Wiener's invariant subspace theorem. For the $p$-adic analogue, my proof makes the Fourier transform arguments work and bypasses the usage of Wiener's invariant subspace theorem; Lee's proof is more direct in nature, which relies on a decomposition of the unity function via the Schwartz--Bruhat function and bypasses the usage of Fourier transform. According to Lee, he did not know that I had obtained my Theorem \ref{image} nor Theorem \ref{equim}; in April 2022,  under the suggestion by Chin-Lung Wang, he abandoned Royden's approach taken by \cite{lee} and adopted Rudin's equimeasurability framework in order to obtain higher dimensional results. Now come back to \cite{lee}. By using the Hasse--Weil bound (cf. \cite{hartshorne} Exercise V 1.10), it is mentioned in \cite{lee} that if $X$ is a smooth geometrically connected curve over $\mathbf F_q$ of genus $g$ and if $q>4g^2$, then $X(\mathbf F_{q})\neq\emptyset$. This raised my interests in finding similar conditions in higher dimensions, with $g$ replaced by intersection numbers involving $K_X$ and polarization divisors, as mentioned above. Part of the outcomes are Theorem \ref{nontrivial} and Corollary \ref{pol sur}, which I wrote in the second version of this paper, also sent to several colleagues around May 2021.  

\section*{Acknowledgement} The $p$-adic case of pseudonorm problems was raised to the author by Professor Shing-Tung Yau about ten years ago. The author would like to thank Professor Yau for having shared with him many problems related to pseudonorms and for constant encouragement. The author also like to thank Professor Dinh Tien-Cuong for showing his interest in the current work during the HAYAMA symposium 2021.  

\section{\bf Integration on $K$-analytic manifolds and pseudonorms}
Let $(K,|\cdot|_K)$ denote a $p$-adic valued field, and let $\mu_{K,n}$ be the associated Haar measure on the locally compact Hausdorff abelian group $K^n$ normalized by setting $\mu_{K,n}(B)=1$ for the unit ball $B=B_1(0)$. We have the following analogue of the fundamental formula for the Lebesgue measure on $\mathbf R^n$  (cf.\ \cite{igusa} 7.4):
\begin{change}\label{change}
Let $\xymatrix{U\ar[r]^-{\varphi}&V}$ be a $K$-bianalytic map between open sets $U$ and $V$ of $K^n$. For any $\mathbf C$-valued Borel measurable function $h$ on $V$ we have
$$
\int_V h\,d\mu_{K,n}= \int_V (h\circ\varphi)\,|\varphi'|_K\,d\mu_{K,n}
$$ 
if one of the two integral exists, where $\varphi'(u)$ denotes the Jacobi matrix of $\varphi$ at the point $u\in U$.
\end{change}

In the following we let $X$ be a $K$-analytic manifold of dimension $n$ (always assumed Hausdorff and second countable). Suppose that $X$ is covered by coordinate patches $(U_i, \varphi_i)\, (i\in I)$ with coordinate systems $u_i=(u_i^1,\dots,u_i^n)$. We may and will assume that $\varphi_i(U_i)$ is compact, since every open ball in $K^n$ is compact. 
\begin{defi} \label{measure 0} We say that a subset $S\subseteq X$ is said to be of measure $0$ in $S$ if $\mu_{K,n}\big(\varphi_i(S\cap U_i)\big)=0$ for every $i\in I$.
\end{defi}
\begin{lem}\label{measure 0 1} Every nowhere dense $K$-analytic subsets of a $K$-analytic manifolds is of measure $0$.  
\end{lem}
\begin{proof}
The question being local completely, we only need to consider the zero set $S$ of a $K$-analytic function $F$ on a compact open subset $U$ of $K^n$. The case $n=1$ is clear since all points of $S$ are isolated. Now we assume $n=2$. For every $p\in U$ we may apply the Weierstrass preparation theorem, which works for arbitrary field equipped with a complete absolute value. By a linear change of coordinate we may find open discs $B\subseteq K^{n-1}$ and $B'\subseteq K$ such that $p\in B\times B'\subseteq U$ and the projection from $S\cap B\times B'$ to $B$ has all its fibres finite. By Fubini's theorem we see that $\mu_{K,n}\big(S\cap (B\times B')\big)=0$. Since $U$ is compact, we have $\mu_{K,n}(S\cap U)=0$.  
\end{proof}

Now we let $K_X$ denote the canonical bundle of $X$, the $K$-analytic line bundle of exterior forms of degree $n$. To each Borel measurable section $\omega$ of $K_X^{\otimes m}, (m\in\mathbf N)$ one may associate canonically a regular\footnote{Here a (positive) Borel measure $\mu$ is said to be regular if 
\begin{itemize}
\item[(i)] $\mu(K)$ is finite for every compact set $K$, 
\item[(ii)] $\mu(B)=\inf\{\mu(U)\,|\,B\subseteq U\text{ and }U\text{ is open in }X\}$ for every Borel set $B$, and 
\item[(iii)] $\mu(U)=\sup\{\mu(K)\,|\,K\subseteq U\text{ and }K\text{ is sompact}\}$.
\end{itemize}
}
 Borel measure $\la\omega\ra^{\frac{1}{m}}$ on $X$. To see this,
we first express $\omega|_{U_i}$ in the form
$$
f_i(u_i^1,\dots,u_i^n)\,(du_i^1\wedge\cdots\wedge du_i^n)^{\otimes m}
$$ 
where $f_i$ is a continuous function on the open set $\varphi_i(U_i)$ of $K^n$. By Riesz's representation theorem, the linear functional
$$
\xymatrix@R3pt{
C_{\mathrm c}(U_i)\ar[r]&\mathbf C\\
h\ar@{|->}[r] & \int_{\varphi_i(U_i)} (h\circ \varphi_i^{-1}) |f_i|^{\frac{1}{m}}\, d\mu_{K,n}
}
$$
defines a regular Borel measure $\mu_i$ on $U_i$. By the above change-of-variable formula we see that $\mu_i|_{U_i\cap U_{i'}}=\mu_{i'}|_{U_i\cap U_{i'}}$ for every pair $(i,i')$. Thus $\mu_i$ patches together to form a regular Borel measure on $X$, which is the expected $\la\omega\ra^{\frac{1}{m}}$. 
\begin{defi} \label{pseudonorm} We define the $m$-th {\it pseudonorm}  
\begin{align*}
\la\!\la\omega\ra\!\ra_m:=\int_X\la\omega\ra^{\frac{1}{m}}\in[0,\infty]
\end{align*}
for every Borel measurable section $\omega$ of $K_X^{\otimes m}$ on $X$. 
\end{defi}
\begin{rmk} \label{measure 0 2} Definition \ref{measure 0} and Lemma \ref{measure 0 1} both have a corresponding version for the measure $\la\!\la\omega\ra\!\ra^{\frac{1}{m}}$.
\end{rmk}

\section{\bf Density of rational points in the set of geometric points}

Let $W$ be an integral algebraic scheme defined over a field $K$. In this section we discuss the density of $W(K)$ in $W(\overline{K})$ with respect to the Zariski topology. The density does not hold in general, as shown by the real algebraic set $\{(x,y)\in\mathbf R^2\,|\,x^2+y^2=0\}$ in the complex one $\{(x,y)\in\mathbf C^2\,|\,x^2+y^2=0\}$. Note that $(0,0)$ is a singular point of the complex algebraic scheme defined by the polynomial $X^2+Y^2=0$. The following density theorem indicates that this phenomenon cannot happen assuming smoothness.
\begin{thm}\label{dens} Let $K$ be a nontrivial complete valued field. Suppose that $W$ is an integral algebraic scheme over $K$ with $\dim\, W>0$. If $x\in W(K)$ is a smooth point\footnote{If an integral algebraic scheme over $K$ has a smooth $K$-point, then it is geometrically integral.} of $W$, then every $K$-analytic neighborhood of $x$ in $W(K)$ is dense in $W(\overline{K})$ with respect to the Zariski topology over $\overline{K}$.
\end{thm}
\begin{proof} Suppose that $W$ is locally defined around $x$ by polynomials $F_i(X_1,\dots, X_r)\in K[X_1,\dots,X_r]\, (i=1,\dots,s)$ in an affine space $\mathbf A^r_K$. Then $x\in W(K)$ is a smooth point if and only if 
$$
\mathrm{rk}_K \left(\frac{\partial F_i}{\partial X_j}(x)\right)=\mathrm{rk}_{\overline K} \left(\frac{\partial F_i}{\partial X_j}(x)\right)=r-\dim W
\quad \text{(the Jacobian criterion).}
$$
In particular, every open neighborhood of $x$ in $W(K)$ is $K$-analytically isomorphic to an open set of $K^n\, (n=\dim\, W)$, by the $K$-analytic implicit function theorem. When $\dim\, W=1$, such a neighborhood is an infinite set, and hence is Zariski dense in the irreducible curve $W(\overline{K})$. We then proceed by induction on $\dim\,W$. Suppose that the statement holds when $\dim\, W<n$ (with $n>1$) and consider the case $\dim\,W=n$. Since the question is local in nature, we may assume that $W$ is a closed subscheme of some projective space $\mathbf P^N_K$. Consider any $K$-analytic coordinate patch $U$ of $W(K)$ around $x$. We claim that $U$ is Zariski dense in $W(\overline{K})$. Were this false, 
$$
U\subseteq V_+(G)\nsupseteq W(\overline{K})
$$ 
for some nonzero homogeneous polynomial $G(Y_0,\dots,Y_N)$ with all coefficients in $K$. For any $d\in\mathbf N$ we let $N_d:={{N+d}\choose{d}}$. For $z=(z_{\alpha_0\cdots\alpha_N})_{\alpha_0+\cdots+\alpha_N=d}\in \overline{K}^{N_d}\setminus \{0\}$ we let 
$$
L_{z}(Y_0,\dots,Y_N)=\sum\limits\nolimits_{\alpha_0+\cdots+\alpha_N=d}z_\alpha Y^\alpha+\cdots+z_NY_N,
$$
$H_z=V_+(L_z)$ be the corresponding hyperplane of $\mathbf P^N_K$, and
$$
P_x:=\left\{[z]\in \mathbf P^{N_d-1}(\overline{K})\,\left|\, x\in H_z(K)\right.\right\}.
$$
Note that $P_x$ is itself a projective hyperplane of $\mathbf P^{N_d}(\overline{K})$ defined over $K$, since $x$ is a $K$-rational point. Consider the sets
$$
V_1:=\left\{[z]\in P_x\,\left|\, W(\overline{K})\nsubseteq H_z(\overline{K})\right.\right\},
$$
$$
V_2:=\left\{[z]\in P_x\,\left|\, \mathbf{T}_xW(\overline{K})\nsubseteq H_z(\overline{K})\right.\right\},
$$
$$
V_3:=\left\{[z]\in P_x\,\left|\, W(\overline{K})\cap H_z(\overline{K})\nsubseteq V_+(G)\right.\right\},
$$
and$$
V_4:=\left\{[z]\in P_x\,\left|\, W_{\overline{K}}\cap (H_z)_{\overline{K}}\text{ is integral}\right.\right\},
$$
where $\mathbf T_xW(\overline{K})$ is the projective tangent space of $W(\overline{K})$ at $x$, and the intersection $W_{\overline{K}}\cap (H_z)_{\overline{K}}$ is to be understood as scheme theoretic. It is direct to verify that  $V_1$, $V_2$, and $V_3$ are all nonempty open sets in $\mathbf P^r(\overline{K})$. By a strengthened version of Bertini's theorem (\cite{gk} Corollary 3.7) we see that, if we choose $d$ sufficiently large, then $V_4$ contains a nonempty open set in $\mathbf P^r(\overline{K})$. Note that $U\cap H_z(K)$ is $K$-analytically smooth at $x$. To see this, note that around $x$ the set of defining equations of $U$ in $\mathbf P^N(K)$ is the same as that of $W(\overline{K})$ in $\mathbf P^N(\overline{K})$. Therefore the condition of the Jacobian criterion for the latter (the algebraic case over $\overline{K}$) at $x$ implies that for the former (the $K$-analytic case).\\ 
{\ul{\bf Claim}.} $P_x\cap \mathbf P^{N_d-1}(K)$ is dense in $P_x$. 

Were the claim proved, one may find some point $[z]\in V_1\cap V_2\cap V_3\cap \mathbf P^{N_d}(K)$. Since $[z]\in \mathbf P^{N_d}(K)$, we have a reduced closed subscheme $W'$ of $W$ such that $W'(\overline{K})=W(\overline{K})\cap H_z(\overline{K})$. 
\begin{itemize}
\item[(1)] Since $[z]\in V_1$, by Krull's Hauptidealsatz $W'(\overline{K})$ has pure dimension $n-1$ and hence $x$ is not an isolated point of $W'(\overline{K})$. 
\item[(2)] Since $[z]\in V_2$, $x$ is a regular point of $W'(\overline{K})$.
\item[(3)] Since $[z]\in V_3$, $V_+(G)\cap W'(\overline{K})$ is a proper closed subset of $W'(\overline{K})$.
\item[(4)] Since $[z]\in V_4$, $W'$ is integral.
\end{itemize}
Therefore, we may apply the induction hypothesis to $W'$. The induction hypothesis implies that $U\cap H_z(K)$ is dense in $W'(\overline{K})$. However, $U\cap H_z(K)$ is contained in the proper closed subset $V_+(G)\cap W'(\overline{K})$ of $W'(\overline{K})$, as contradicts to its density in $W'(\overline{K})$. 

It remains to verify the claim. Since $P_x$ is a projective hyperplane of $\mathbf P^{N_d-1}(\overline{K})$ defined over $K$, it suffices proving the following statement: $\mathbf P^k(K)$ is dense in $\mathbf P^k(\overline{K})$ with respect to the Zariski topology. Suppose that $\mathbf P^k(K)$ is not dense in $\mathbf P^k(\overline{K})$, say, $\mathbf P^k(K)\subseteq V_+(F)$ for some nonzero homogeneous polynomial 
$$
F(Z_0,\dots, Z_k)=\sum_{\alpha}F_\alpha Z^\alpha\in\overline{K}[Z_0,\dots,Z_k].
$$ 
Fixing a basis $e_s\, (s\in S)$ of $\overline{K}$ over $K$, we may write
$$
F(Z_0,\dots, Z_k)=\sum_{s\in S_0}\Big( e_s\sum_{\alpha}F_{s,\alpha} Z^\alpha\Big)
$$
with $S_0$ finite and $F_{s,\alpha}\in K$ for all $(s,\alpha)$. The condition $\mathbf P^k(K)\subseteq V_+(F)$ is equivalent to that $\sum_\alpha F_{s,\alpha}b_\bullet^\alpha=0$ for every $s\in S_0$ if $(b_0,\dots,b_k)\in K^{k+1}\setminus\{0\}$. Since $K$ is infinite, this implies that $F_{s,\alpha}=0$ for every $(s,\alpha)$, and hence $F=0$, a contradiction.
\end{proof}

\begin{rmk} Theorem \ref{dens} has a simpler proof using measure theory when $K$ is a $p$-adic field and $W$ is proper and smooth over its ring of integers $R$. This is the only situation we need for the rest of the paper. If this is the case, for any closed subset $Z\nsubseteq W$, it is well known that $Z(K)=Z(R)$ is of measure $0$ in the $K$-analytic manifold $W(K)=W(R)$. On the other hand all nonempty open subsets $W(K)$ has strictly positive measure. Then we obtain the expected density by working on affine charts.  
\end{rmk}

\section{\bf $R$-birational equivalence and pseudonorms}
In this section we suppose that $K$ is a $p$-adic field and $R$ its ring of integers. Consider an $R$-birational map $\xymatrix{X\ar@{-->}[r]^-f&Y}$ between smooth $R$-schemes. We recall the following easy version of global change-of-variable formula.
\begin{lem}\label{global change} If $f$ is an $R$-morphism, we have $\la\!\la\omega\ra\!\ra_m=\la\!\la f^*\omega\ra\!\ra_m$ for any $\omega\in \Gamma(Y,\omega^{\otimes m}_{Y/R})$.
\end{lem}
\begin{proof} There exists a dense open subset $V$ of $Y$ such that $f$ is an $R$-isomorphism from $f^{-1}(V)$ to $V$. By the local change-of-variable formula \ref{change} we have 
$$
\int_{V(K)}\la\omega\ra^{\frac{1}{m}}=\int_{f^{-1}\big(V(K)\big)}\la f^*\omega\ra^{\frac{1}{m}}.
$$
Let $Z=Y\setminus V$ and $Z'=X\setminus f^{-1}(Z)$. Then $Z(K)$ and $Z'(K)$ are $K$-analytic subsets of the $K$-analytic manifolds $Y(K)$ and $X(K)$, respectively. Theorem \ref{dens} implies that they are nowhere dense, and hence are of measure $0$ in their ambient manifolds by Remark \ref{measure 0 2}. Therefore, 
$$
\la\!\la\omega\ra\!\ra_m=\int_{Y(K)}\la\omega\ra^{\frac{1}{m}}=\int_{V(K)}\la\omega\ra^{\frac{1}{m}}=\int_{f^{-1}\big(V(K)\big)}\la f^*\omega\ra^{\frac{1}{m}}=\int_{X(K)}\la f^*\omega\ra^{\frac{1}{m}}=\la\!\la f^*\omega\ra\!\ra_m
$$
\end{proof}
The pull-back map $f^*$ extends to the $R$-birational case. Moreover, we have: 
\begin{thm} \label{bir inv} If both $X$ and $Y$ are proper over $\mathrm{Spec}\, R$, then for every $m\in\mathbf N$ there is a canonical $R$-linear isometry
$$
\xymatrix{\big(\Gamma(Y,\omega_{Y/R}^{\otimes m}), \la\!\la\cdot\ra\!\ra_m\big)\ar[r]^-{f^*}&\big(\Gamma(X,\omega_{X/R}^{\otimes m}), \la\!\la\cdot\ra\!\ra_m\big).}
$$  
\end{thm}
\begin{proof} By Hironaka's desingularization theorem \cite{hironaka} there exists a common resolution $W$ which is projective over $R$ and is compatible with $f$: 
$$
\xymatrix{&W\ar[ld]_g \ar[rd]^-h&\\
X\ar@{-->}[rr]_-f && Y.
} 
$$
Note that $\la\!\la\cdot\ra\!\ra_m$ takes finite values on $\Gamma(W,\omega_{W/R}^{\otimes m})$ since $W$ is projective over $R$. Thus $\la\!\la\cdot\ra\!\ra_m$ takes finite values on both $\Gamma(X,\omega_{X/R}^{\otimes m})$ and $\la\!\la\cdot\ra\!\ra_m$ takes finite values on $\Gamma(Y,\omega_{Y/R}^{\otimes m})$ by Lemma \ref{global change}. It suffices to establish the isometries for $g$ and $h$, i.e., we may assume $f$ an $R$-morphism. Therefore, there exists an open subset $V$ of $Y$ with $\mathrm{codim}_Y(Y\setminus V)\geqslant 2$ such that $f$ is an $R$-isomorphism from $f^{-1}(V)$ to $V$. We have 
a unique $R$-linear maps $f^*$ fitting in the following commutative diagram:
$$
\xymatrix@R40pt@C60pt{
\Gamma(Y,\omega_{Y/R}^{\otimes m})\ar[d]_-{\rm injective}^-{(\cdot)|_{V}}\ar@{.>}[r]^-{f^*}
&
\Gamma(X,\omega_{X/R}^{\otimes m})\ar[d]_-{\simeq}^-{(\cdot)|_{f^{-1}(V)}}\\
\Gamma(V,\omega_{Y/R}^{\otimes m})
\ar[r]^-{(f|_{f^{-1}(V)})^*}_-{\simeq}&
\Gamma\big(f^{-1}(V),\omega_{X/R}^{\otimes m}\big).
}
$$
By lemma \ref{global change} all the above maps preserves $\la\!\la\cdot\ra\!\ra_m$. It is direct to see that $(f^{-1})^*=(f^*)^{-1}$ ($f^{-1}$ being the inverse $R$-rational map of $f$).
\end{proof}

\section{\bf $p$-adic analogue of Rudin's equimeasurability theorem}
   
\begin{defi}\label{equim def}
Let $(X,\mathfrak{A},\mu)$ and $(Y,\mathfrak{B},\nu)$ be measured spaces, $(Z,\mathfrak{C})$ be a measurable space, and   
$$
\xymatrix{&Z&\\
X\ar[ru]^-F&&Y\ar[lu]_-G
}
$$
be measurable maps, i.e., $F^{-1}(C)\in\mathfrak{A}$ and $G^{-1}(C)\in\mathfrak{B}$ for every $C\in\mathfrak{C}$. We say that $f$ and $g$ are equimeasurable if $F_*\mu = G_*\nu$, i.e., $\mu\big(F^{-1}(C)\big)=\mu\big(G^{-1}(C)\big)$ for every $C\in\mathcal C$.
\end{defi}

The following is a $p$-adic analogue of a theorem due to Rudin (\cite{rudin} 1.4 Theorem I). 
\begin{thm}\label{equim} 
Suppose that $(K, |\cdot|_K)$ is a $p$-adic field and view $K^n$ as the measurable space equipped with the $\sigma$-algebra of all Borel subsets. Let $r\in \mathbf Q_{>0}$. For any measurable maps 
$$
\xymatrix{&K^n&\\
X\ar[ru]^-{F=(f_1,\dots,f_n)}&&Y\ar[lu]_-{G=(g_1,\dots,g_n)}
}
$$ 
from measured spaces $(X,\mathfrak{A},\mu)$ and $(Y,\mathfrak{B},\nu)$ to $K^n$, if 
$$
f_i\in \mathcal L^1(X,\mathfrak{A},\mu)
\,\text{ and }\, 
g_i\in \mathcal L^1(Y,\mathfrak{B},\nu)
\,\text{ for }\, 
i=1,\dots,n
$$ 
and
\begin{align}\label{p isometry}
\int_X \big|1+v_1f_1(x)+\cdots+v_1f_1(x)\big|_K^r d\mu 
=
\int_Y \big|1+v_1g_1(x)+\cdots+v_1g_1(x)\big|_K^r d\nu
\end{align}
for every $(v_1,\dots,v_n)\in K^n$, then $F$ and $G$ are equimeasurable.
\end{thm}
  Before giving the proof we first make some formal settings. Let $(S,\mathcal M)$ be a measurable space and $M(S,\mathfrak{M})$ the space of all complex measures on $(S,\mathfrak{M})$. For any family $\mathcal F$ of $\mathfrak{M}$-measurable functions on $S$ we let   
$$
\mathcal F^\perp:=
\left\{
m\in M(S,\mathfrak{M})
\,\left|\,
\begin{array}{c}
f\text{ is integrable with respect to }|m|\text{ and }\\
\int_S f(s)\,dm(s)=0\ \text{ for every }f\in \mathcal F
\end{array}
\right.
\right\}.
$$
On the other hand, for any family $\mathcal M\subseteq M(S,\mathfrak{M})$ we let
$$
\mathcal M^\perp:=
\left\{
f\text{ : }\mathfrak{M}\text{-measurable}\,\left|\,
\begin{array}{c}
f\text{ is integrable with respect to }|m|\text{ and }\\
\int_S f(s)\,dm(s)=0\ \text{ for every }m\in \mathcal M
\end{array}
\right.
\right\}.
$$
For any $\mathfrak{M}$-measurable map $\xymatrix{S\ar[r]^-T&S}$ we let
$$
T_*m:=m\circ T^{-1}\text{ for every }m\in B(S,\mathfrak{M})
$$ 
and
$$
T_*f:=f\circ T\text{ for every }\mathfrak{M}\text{-measurable function }f.
$$ 
Then 
$$
\int_S T^*f\, dm = \int_S f\, d(T_*m)
$$
if $f$ is integrable with respect to $|m|$. We have $T_*|m|=|T_*m|$ if $T$ is further assumed bijective. The following lemma is easily derived from the definitions.
\begin{lem}\label{T-inv} Let $\xymatrix{S\ar[r]^-{T}&S}$ be a bijective $\mathfrak{M}$-measurable map. For any family $\mathcal F$ of $\mathfrak{M}$-measurable functions on $S$ (resp.\ any $\mathcal M\subseteq M(S,\mathfrak{M})$), we have 
$$
T^*\mathcal F\subseteq \mathcal F
\, \Rightarrow\, 
T_*(\mathcal F^{\perp})\subseteq\mathcal F^{\perp}
\quad
\Big(\text{resp.\quad }
T_*\mathcal M\subseteq \mathcal M
\, \Rightarrow\, 
T^*(\mathcal M^{\perp})\subseteq\mathcal M^{\perp}
\Big).
$$
\end{lem} 

\begin{proof}[{\bf Proof of Theorem \ref{equim}}]
We let $(S,\mathfrak{M})$ be $K$ with its $\sigma$-algebra of Borel subsets and 
$$
\mathcal F=
\Big\{\,
|at+b|_K^r\,\Big|\, a\in K^\times,\, b\in K
\Big\}.
$$ 
Let $(u_1,\dots, u_n)$ be the coordinate system of $K^n$. For every $v=(v_1,\dots, v_n)\in K^n$ we consider ``projection map'' $\xymatrix{K^n\ar[r]^-{L_v}&K}$ with $L_v(u_1,\dots,u_n)=v_1u_1+\dots+v_nu_n$. We set 
$$
m_v=(L_v)_*(F_*\mu-G_*\nu).
$$
Then (\ref{p isometry}) says nothing but
\begin{align}\label{p isometry 2}
m_v\in\mathcal F^{\perp}\, \text{ for every }\, v\in K^n. 
\end{align}
Now we set 
$$
\mathcal F':=(\mathcal F^\perp)^\perp\cap \mathcal L^1(K,ds)\quad (ds\text{ being the normalized Haar measure}).
$$
Since $\mathcal F$ is invariant under all affine transformations, so is $\mathcal F^\perp$ and $(\mathcal F^\perp)^\perp$ by Lemma \ref{T-inv}. Since $\mathcal L^1(K, ds)$ is also invariant under all affine transformations,\footnote{The invariance of $\mathcal L^1(K,ds)$ under dilations comes from the change-of-variable formula \ref{change}.} so is $\mathcal F'$ Therefore, for any $f\in\mathcal F'$ and any $m\in\mathcal F^\perp$, we have
$$
(f\ast m)(t) = \int_K f(t-s)\, dm(s) =0\, \text{ for every }\, t\in K.
$$
Applying the above argument to $m=m_v$ and taking the Fourier transform yield that
\begin{align}\label{fourier}
\widehat f\,\widehat m_v=0\, \text{ for every }\, f\in \mathcal F'\, \text{ and for every }\, v\in K^n
\end{align}
where we identify the character group $\widehat{K}$ with $K$ as topological groups. More precisely, let $\Lambda$ be any fixed nontrivial element of $\widehat K$, the pairing (cf. \cite{tate} Theorem 2.2.1)
$$
\xymatrix@R3pt{K\times K\ar[r]&\mathbf S^1\\
(\tau, t)\ar@{|->}[r]& \exp 2\pi i\Lambda(\tau t)
}
$$
gives an isomorphism $K\simeq \widehat K$.
\\
{\ul{\bf Claim}.} $\widehat m_v=0$ (and hence $m_v=0$) for every $v\in K^n$. 
\\
Assume the claim for now. By (\ref{p isometry 2}) we have
$$
m_v=0\, \text{ for every }\, v\in K^n. 
$$
This implies $\widehat{F_*\mu-G_*\nu}=0$ (and hence $F_*\mu-G_*\nu=0$) as follows. We have similar identifications $K^n\simeq \widehat{K}^n\simeq\widehat{ K^n}$ via the pairing)
$$
\xymatrix@R3pt{K^n\times K^n\ar[r]&\mathbf S^1\\
(v_\bullet, u_\bullet)\ar@{|->}[r]& \exp 2\pi i\Lambda\big(-\sum\nolimits_j v_j u_j\big)
}
$$
\begin{align*}
\widehat{F_*\mu-G_*\nu}(v_1,\dots,v_n)
&=
\int_{K^n}
\exp 2\pi i\Lambda \big(-\sum\nolimits_j v_j u_j\big)\, d(F_*\mu-G_*\nu)(u_\bullet) \\
&=
\int_K \exp 2\pi i\Lambda (-t)\, dm_v(t) =\widehat{m_v}(1) =0.
\end{align*}
It remains proving the claim. Note that by (\ref{p isometry}) with $(z_1,\dots,z_n)=(0,\dots,0)$ we have
\begin{align*}
\widehat m_v(0)&=\int_K dm_v = \int_{K^n} d(F_*\mu-G_*\nu)= \int_X d\mu  - \int_Y d\nu =0.
\end{align*}
By (\ref{fourier}), it suffices to show that for every $\tau\in K^\times$ there exists some $f\in\mathcal F'$ with $\widehat f(\tau)\neq 0$. Actually it will be enough to obtain a function $f_0\in\mathcal F'$ which is not identically zero: for such $f_0$ we have $\widehat{f_0}$ nontrivial, say $\widehat{f_0}(\tau_0)\neq 0$ for some $\tau_0\in K^\times$. Then \begin{align*}
0\neq \widehat{f_0}(\tau_0) 
& = \int_K \exp 2\pi i\Lambda (-\tau_0 s) f_0(s)\, ds
\\
&= \int_K \exp 2\pi i\Lambda (-\tau s) f_0(\tau\tau_0^{-1} s)\,|\tau\tau_0^{-1}|_K\, ds
= \widehat f (\tau)
\end{align*}  
where $f(s):=f_0(\tau\tau_0^{-1}s)\,|\tau\tau_0^{-1}|_K$. Since $\mathcal F'$ is invariant under dilation, we see that $f\in\mathcal F'$. To construct such a nontrivial $f_0\in \mathcal F'$, we consider functions of the form
\begin{align}\label{f0}
f_0(s)=\sum\limits_{k=1}^N a_k |1+ c_k s|_K^r
\end{align}  
where $a_1,\dots,a_N\in \mathbf R$ and $c_1,\dots,c_N\in K^\times$ are to be chosen. We first try to analyse the behaviour of $f_0(s)$ when $|s|_K$ is sufficiently large. Let $b_j={r\choose j}\in\mathbf Q$ be the binomial coefficients. Then
$$
(1+\sigma)^r :=\sum\limits_{j=1}^\infty b_j\sigma^j\, \text{ converges in }K\text{ when }\, |\sigma|_K<1
$$ 
since $|b_j|_K\leqslant 1$ for all $j\in\mathbf N$. Moreover, we have
$$
|(1+\sigma)^r|_K=|1+\sigma|_K^r\, \text{ when }\, |\sigma|_K<1.
$$
This can be see as follows. Write $r=\frac{l}{m}$ with $l,m\in\mathbf N$. When $|\sigma|_K<1$, by multiplying power series we see that $\big((1+\sigma)^r\big)^m = (1+\sigma)^l$, and hence 
$$ 
\big|(1+\sigma)^r\big|_K=\big|(1+\sigma)^r\big|_K^{m\frac{1}{m}}=\big|\big((1+\sigma)^r\big)^m\big|_K^{\frac{1}{m}} = \big|(1+\sigma)^l\big|_K^{\frac{1}{m}}=\big|1+\sigma\big|_K^{\frac{l}{m}}=\big|1+\sigma\big|_K^r.
$$
Therefore,
$$
|1+s|^r_K=|s|^r_K\left|\Big(1+\frac{1}{s}\Big)^r\right|_K\, \text{ when }\,|s|_K>1.
$$
From now on we denote $|\cdot|_K$ by $|\cdot|$. When $|s|>\max\{1,|c_1|,\dots,|c_N|\}$, we have
\begin{align*}f_0(s)&=\sum\limits_{k=1}^N \left(a_k\, |c_k|^r\, |s|^r \, \left|\sum\limits_{j=1}^\infty b_j\,c_k^{-j} s^{-j}\right|\right)
\\
&=\sum\limits_{k=1}^N a_k\, |c_k|^r\, |s|^r \, 
\max
\left\{
|b_j|\,
|c_k|^{-j} \,
|s|^{-j}
\,\Big|\, j\in\mathbf N\right\}.
\end{align*}
Note that 
$$
1=|b_0|\, |c_k|^0\,|s|^0 > |b_j|\, |c_k|^{-j}\,|s|^{-j}
\, \text{ if and only if }\,  
|s|_K> \frac{|b_j|^{\frac{1}{j}}_K}{|c_k|}.
$$
In summary, 
$$
f_0(s)=\left(\sum\limits_{k=1}^N a_k\, |c_k|^r\right) |s|^r
$$
when 
$$
|s|>\max
\left\{
1,|c_1|,\dots,|c_N|,\, 
\max\limits_{j\in\mathbf N}\,|b_j|^{\frac{1}{j}}\,\max\{\,|c_k|^{-1}\,|\,k=1,\dots, N\}
\right\}.
$$
If $a_k$ and $c_k\, (k=1,\dots,N)$ are chosen so that 
\begin{align}\label{condition 1}
\sum\limits_{k=1}^N a_k\, |c_k|^r =0,
\end{align}
then $f_0$ vanishes outside a sufficiently big ball, and hence is bounded and continuous. In particular, $f_0\in \mathcal L^1(K,ds)$. Besides, $f_0\in\mathcal F\subseteq (\mathcal F^\perp)^\perp$, and hence $f_0\in\mathcal F'$. It remains to make $f_0$ nontrivial. Since $f_0(0)=\sum\limits_{k=1}^N a_k$, it suffices to set $N=2$ and choose $c_1, c_2\in K^\times$ with $|c_1|_K\neq |c_2|_K$ and $a_1,a_2$ satisfying (\ref{condition 1}). 
\end{proof}   

\section{\bf Proof of Theorem \ref{image}}\label{proof of image}

Recall that the $R$-linear isometry  
$$
\xymatrix{\big(V_Y, \la\!\la\cdot\ra\!\ra_m\big)\ar[r]^-{T}&\big(V_X, \la\!\la\cdot\ra\!\ra_m\big)}
$$ 
induces an $R$-isomorphism
$$
\gamma\,:=\,\mathbf PT^*:\xymatrix{\mathbf P V_X^*\ar[r]^-\simeq& \mathbf P V_Y^*}.
$$
Let $X'$ denote the image closure of $\varphi_{|V_X|}$, which is defined over $R$. $X'$ is an integral schemes of finite type over $R$. Similar setting will be adopted for $Y$. In order to show that $\mathbf PT^*$ identifies $X'$ with $Y'$, it suffices to prove the statement over $\overline{K}$-points: \\
\ \\
{\ul{\bf Claim 1}.} If $\gamma\big(X'(\overline{K})\big)=Y'(\overline{K})$, then $X'$ is identified with $Y'$ via $\gamma$.

$\gamma(X')$ and $Y'$ are integral closed $R$-subschemes of the projective scheme $\mathbf P V_Y^*\simeq \mathbf P^N_R$. Suppose that $\gamma(X')$ and $Y'$ are locally defined on an affine open set $U\simeq\mathrm{Spec}\, A$ of $\mathbf P^N_R$  by ideals $I$ and $J\unlhd A$. Since $\overline{K}$ is algebraically closed and $\gamma\big(X'(\overline{K})\big)=Y'(\overline{K})$, we have $I\otimes_R \overline{K}=J\otimes_R \overline{K}$. Note that
$$
I= A\cap (I\otimes_R\overline{K})\quad  (\text{all viewed as subsets of } A\otimes_R\overline{K})
$$
as can be seen by the following commutative diagram with exact rows:
$$
\xymatrix{
0\ar[r] & I\otimes_R\overline{K}
\ar[r] & A\otimes_R\overline{K}\ar[r] & (A/I)\otimes_R\overline{K}\ar[r] & 0\\
0\ar[r] & I\ar@{^{(}->}[u]\ar[r] & A\ar@{^{(}->}[u]\ar[r] & A/I\ar@{^{(}->}[u]\ar[r] & 0
}.
$$ 
Similarly, $J= A\cap (J\otimes_R\overline{K}) =A\cap (I\otimes_R\overline{K})=I$.\\
\ \\
{\ul{\bf Claim 2}.} $Y'(\overline{K})\subseteq \gamma\big(X'(\overline{K})\big)$.

Fix an $R$-basis of $\eta_0,\dots,\eta_N\in \Gamma(Y,\omega^{\otimes m}_{Y/R})$. Then $(T\eta_0)_K,\dots, (T\eta_N)_K$ form a basis of $V_{X,K}$. The rational maps $\gamma\circ\varphi_{|V_X|}$ and $\varphi_{|V_Y|}$ can be viewed as mapping some nonempty open sets $X_0\subseteq X$ resp.\, $Y_0\subseteq Y$, respectively, to the affine $R$-scheme $\mathbf A^N_R$, given by 
$$
\left(\frac{T\eta_1}{T\eta_0},\dots, \frac{T\eta_N}{T\eta_0}\right) 
\,\text{ and }\,
\left(\frac{\eta_1}{\eta_0},\dots, \frac{\eta_N}{\eta_0}\right),
$$
which determine maps
$$
\xymatrix@C70pt{&\overline{K}^N&\\
X_0(\overline{K})\ar[ru]^-{F_{\overline{K}}}&&Y_0(\overline{K})\ar[lu]_-{G_{\overline{K}}}\\
&K^N\ar@{^{(}->}[uu]&\\
X_0(K)\ar@{^{(}->}[uu]\ar[ru]^-{F=(f_1,\dots,f_N)\ }&&Y_0(K)\ar@{^{(}->}[uu]\ar[lu]_-{\ G=(g_1,\dots,g_N)}}.
$$ 
We let $\mu$ and $\nu$ be the measures on $X$ and on $Y$ associated to $\la T\eta_0\ra^{\frac{1}{m}}$ and $\la\eta_0\ra^{\frac{1}{m}}$, respectively. Then 
for every $(v_1, v_2,\dots, v_N)\in K^N$ we have
$$
\Big\la T\Big(\eta_0+\sum\nolimits_{i=1}^N v_i\eta_i\Big)\Big\ra^{\frac{1}{m}}
=
\Big|1+\sum\nolimits_{j=1}^N v_i f_i\Big|_K^{\frac{1}{m}} \, \la T\eta_0\ra^{\frac{1}{m}}
$$
and
$$ 
\Big\la \eta_0+\sum\nolimits_{i=1}^N v_i \eta_i\Big\ra^{\frac{1}{m}}
=
\Big|1+\sum\nolimits_{i=1}^N v_i g_i\Big|_K^{\frac{1}{m}} \, \la \eta_0\ra^{\frac{1}{m}}.
$$
Since $X$ and $Y$ are proper over $R$, we have
$$ 
\int_{X_0(K)}\la T(\eta_i)\ra^{\frac{1}{m}}
=\int_{X(K)}\la T(\eta_i)\ra^{\frac{1}{m}}<\infty
\,\text{ and }\,
\int_{Y_0(K)}\la \eta_i\ra^{\frac{1}{m}}
=\int_{Y(K)}\la \eta_i\ra^{\frac{1}{m}}<\infty,
$$
or equivalently, 
$$
f_i\in \mathcal L^1\big(X_0(K),\mathfrak{B}_{X_0},\mu\big)
\,\text{ and }\, 
g_i\in \mathcal L^1\big(Y_0(K),\mathfrak{B}_{Y_0},\nu\big)
\,\text{ for }\, 
i=1,\dots,N,
$$ 
where $\mathfrak{B}_{X_0}$ and $\mathfrak{B}_{Y_0}$ are the corresponding $\sigma$-algebra of Borel sets. The isometry condition on
$
\xymatrix{\big(V_Y, \la\!\la\cdot\ra\!\ra_m\big)\ar[r]^-{T}&\big(V_X, \la\!\la\cdot\ra\!\ra_m\big)}
$
is then equivalent to that
$$
\int_{X_0(K)} \Big|1+\sum\nolimits_{j=1}^N v_i f_i\Big|_K^{\frac{1}{m}} \, \la T\eta_0\ra^{\frac{1}{m}}
=
\int_{Y_0(K)}
\Big|1+\sum\nolimits_{i=1}^N v_i g_i\Big|_K^{\frac{1}{m}} \, \la \eta_0\ra^{\frac{1}{m}}
$$
for every $(v_1, v_2,\dots, v_N)\in K^N$. By Theorem \ref{equim} we conclude that $F$ and $G$ are equimeasurable (Definition \ref{equim def}). Now suppose that $Y'(\overline{K})\nsubseteq\gamma\big(X'(\overline{K})\big)$. We will show that the equimeasurability of $F$ and $G$ is violated. By Chevalley' theorem, there exists a nonempty Zariski open subset $U$ of $Y'\setminus \gamma(X')$ such that 
$$
\varphi_{|V_Y|}^{-1}(U)\subseteq Y_0 
\, \text{ and }\,
U=\varphi_{|V_Y|}\big(\varphi_{|V_Y|}^{-1}(U)\big)=G\big(\varphi_{|V_Y|}^{-1}(U)\big).
$$ 
We denote $\varphi_{|V_Y|}^{-1}(U)$ by $W$. Since $Y$ is smoorh and $Y(K)$ is assumed nonempty, Theorem \ref{dens} implies that $W(K)\neq\emptyset$, and hence it contains a nonempty open subset of the $K$-analytic manifold $Y(K)$. On the other hand, $F^{-1}(U)=\emptyset$ since $U\cap \gamma(X')=\emptyset$. Therefore
$$
0=\mu\big(F^{-1}\big(U(K)\big)\big)=\nu\big(G^{-1}\big(U(K)\big)\big)=\int_{W(K)}\la\eta_0\ra^{\frac{1}{m}}>0,
$$ 
a contradiction.

\section{\bf Existence of rational points over finite fields}

As indicated in the introduction, the nontriviality of $X(K)$ comes from that of $X(\mathbf F_q)$.

In this section we prove Theorem \ref{nontrivial}. First recall the Hasse--Weil bound when $W$ is a smooth geometrically connected complete curve of genus $g$: 
$$
\left|1+q - \# W\big(\mathbf F_{q})\big)\right| \leqslant 2g\sqrt{q}
$$ 
(see for example \cite{hartshorne} Exercise V 1.10). This implies that$W(\mathbf F_{q})\neq\emptyset$ if $q>4g^2$. 

\begin{thmrp} [{{\bf Theorem \ref{nontrivial}}}]
Let $W$ be a complete smooth geometrically connected scheme over $\mathbf F_q$ of dimension $n$. Suppose that $H$ is a very ample divisor. Then 
$$
W(\mathbf F_q)\neq\emptyset\ \text{ if }\ q>\max\left\{H^{\bullet n}(H^{\bullet n}-1)^n,\, 
\big(K_W\bullet H^{\bullet (n-1)} + (n-1) H^{\bullet n}+2\big)^2
\right\}.
$$
\end{thmrp} 
\begin{proof} 
In the following we let $\overline{(\cdot)}$ denote the operation of base-change induced by a fixed extension $\mathbf F_q\rightarrow\overline{\mathbf F}_q$, and objects without an overline are defined over $\mathbf F_q$. We may assume that $n\geqslant 2$. The very ample divisor $H$ gives an embedding of $W$ into a projective space $\mathbf P^N_{\mathbf F_q}$. We let $d=H^{\bullet n}$ be the degree of $W$ in $\mathbf P^N_{\mathbf F_q}$. Since $q>d(d-1)^n$, there exists a hyperplane $H_1$ in $\mathbf P^N_{\mathbf F_q}$ such that $\overline{H}_1$ intersects $\overline{W}$ transversally (\cite{ballico} Theorem 1). Since $n=\dim\, \overline{W}\geqslant 2$, we see that $\overline{W}\cap \overline{H}_1$ is connected (\cite{hartshorne} Corollary 7.9). We let $W_1=W\cap H_1$. Then $W_1$ has degree $d$ in $\mathbf P^N_{\mathbf F_q}$, too, and $q>d(d-1)^n\geqslant d(d-1)^{n-1}$. We may repeat the argument to find hyperplane $H_1, \dots, H_{n-1}$ in $\mathbf P^N_{\mathbf F_q}$ such that $\overline{W}_j=\overline{W}_{j-1}\cap \overline{H}_j$ is smooth and geometrically connected. In particular, $W_{n-1}$ is a complete smooth geometrically connected curve over $\mathbf F_q$. By the adjunction formula, 
\begin{align*}
2g(W_{n-1})-2
&= (K_W+H_1+\cdots+H_{n-1})\bullet H_1\bullet\cdots\bullet H_{n-1}\\
&=K_W\bullet H^{\bullet (n-1)} + (n-1) H^{\bullet n}.
\end{align*}
By the Hasse--Weil bound, we see that $W(\mathbf F_q)\neq \emptyset$ since
\begin{align*}
q> \Big(K_W\bullet H^{\bullet (n-1)} + (n-1) H^{\bullet n}+2\Big)^2.
\end{align*}
This completes the proof.
\end{proof}

\end{document}